\documentclass[11pt]{article}
\usepackage{amsmath,amsthm,amssymb,stackengine}
\usepackage[hyperfootnotes=false]{hyperref}
\usepackage{color}
\usepackage{cancel}
\usepackage{titlesec}

\usepackage{float}

\titleformat{\paragraph}
{\normalfont\normalsize\bfseries}{\theparagraph}{1em}{}
\titlespacing*{\paragraph}
{0pt}{3.25ex plus 1ex minus .2ex}{1.5ex plus .2ex}

\def\titlerunning#1{\gdef\titrun{#1}}
\makeatletter
\def\author#1{\gdef\autrun{\def\and{\unskip, }#1}\gdef\@author{#1}}
\def\address#1{{\def\and{\\\hspace*{18pt}}\renewcommand{\thefootnote}{}%
\footnote {#1}}%
\markboth{\autrun}{\titrun}}
\makeatother
\def\email#1{\hspace*{4pt}{\em e-mail}: #1}
\def\MSC#1{{\renewcommand{\thefootnote}{}%
\footnote{\emph{Mathematics Subject Classification (2020):} #1}}}
\def\keywords#1{\par\medskip
\noindent\textbf{Keywords:} #1}

\newtheorem{theorem}{Theorem}[section]

\newtheorem{prop}[theorem]{Proposition}
\newtheorem{cor}[theorem]{Corollary}
\newtheorem{lemma}[theorem]{Lemma}

\newtheorem{defin}[theorem]{Definition}

\theoremstyle{definition}

\newtheorem{prob}[theorem]{Problem}

\newtheorem{remark}[theorem]{Remark}

\numberwithin{equation}{section}

\def\s{\mathfrak s}

\def\cA{\mathcal A}

\def\cC{\mathcal C}

\def\cE{\mathcal E}

\def\cH{\mathcal H}

\def\cL{\mathcal L}

\def\cO{\mathcal O}
\def\cP{\mathcal P}
\def\cQ{\mathcal Q}

\def\cR{\mathcal R}

\def\cT{\mathcal T}
\def\cU{\mathcal U}

\def\cW{\mathcal W}
\def\cX{\mathcal X}

\def\PG{{\rm PG}}

\def\F{{\mathbb F}}

\def\PGL{{\rm PGL}}

\def\PSp{{\rm PSp}}

\def\PGU{{\rm PGU}}

\def\i{\boldsymbol \iota}
\def\y{y_{q^2-q+1}}
\def\z{y_{q^2+1}}
\def\t{y_{q^2+q+1}}

\begin{document}


\baselineskip=16pt

\titlerunning{}

\title{On large partial ovoids of symplectic and Hermitian polar spaces}

\author{Michela Ceria
\and
Jan De Beule
\and 
Francesco Pavese
\and 
Valentino Smaldore
}

\date{}

\maketitle
\address{M. Ceria, F. Pavese: Department of Mechanics, Mathematics and Management, Polytechnic University of Bari, Via Orabona 4, 70125 Bari, Italy; \email{\{michela.ceria, francesco.pavese\}@poliba.it}
\and
J. De Beule: Department of Mathematics and Data Science, Vrije Universiteit Brussel, Pleinlaan 2, 1050 Brussel, Belgium and Department of Mathematics: Analysis, Logic and Discrete Mathematics, Ghent University, Krijgslaan 281 (S8), 9000 Gent, Belgium; \email{jan@debeule.eu}, \url{https://orcid.org/0000-0001-5333-5224}
\and
V. Smaldore: Department of Mathematics, Computer Science and Economics, University of Basilicata, Contrada Macchia Romana, 85100, Potenza, Italy; \email{valentino.smaldore@unibas.it}
}


\MSC{Primary 51E20. Secondary 05B25.}

\begin{abstract}
In this paper we provide constructive lower bounds on the sizes of the largest partial ovoids of the symplectic polar spaces $\cW(3, q)$, $q$ odd square, $q \not\equiv 0 \pmod{3}$, $\cW(5, q)$ and of the Hermitian polar spaces $\cH(4, q^2)$, $q$ even or $q$ odd square, $q \not\equiv 0 \pmod{3}$, $\cH(6, q^2)$, $\cH(8, q^2)$.

\end{abstract}

\section{Introduction}

Let $\cP$ be a finite classical polar space, i.e., $\cP$ arises from a vector space of finite dimension over a finite field equipped with a non-degenerate reflexive sesquilinear or quadratic form. A projective subspace of maximal dimension contained in $\cP$ is called a {\em generator} of $\cP$. A {\em (partial) ovoid} $\cO$ of a polar space $\cP$ is a set of points of $\cP$ such that every generator of $\cP$ contains (at most) one point of $\cO$. A partial ovoid is said to be {\em maximal} if it is maximal with respect to set-theoretic inclusion. 

Here we focus on the symplectic polar space $\cW(2n-1, q)$ and the Hermitian polar space $\cH(2n, q^2)$, consisting of the absolute projective subspaces with respect to a non-degenerate symplectic polarity of $\PG(2n-1, q)$ and a non-degenerate unitary polarity of $\PG(2n, q^2)$, respectively. It is well known that $\cW(2n-1, q)$ has ovoids if and only if $n =2$ and $q$ is even, whereas $\cH(2n, q^2)$ does not possess ovoids. Hence the question of the largest (maximal) partial ovoids of $\cW(2n-1, q)$, $(n, q) \ne (2, 2^h)$, and $\cH(2n, q^2)$ naturally arises. Constructions of maximal partial ovoids of $\cW(3, q)$, $q$ even, of size an integer between about $q^2/10$ and $9q^2/10$ or of size $q^2-hq+1$, $1 \le h \le q/2$, have been provided in \cite{RS, T}. The situation is somewhat different for $q$ odd where the lack of examples is transparent. In \cite{T} Tallini proved that a partial ovoid of $\cW(3, q)$, $q$ odd, has size at most $q^2-q+1$ and constructed a maximal partial ovoid of $\cW(3, q)$ of size $2q+1$. Regarding symplectic polar spaces in higher dimensions an upper bound on the size of the largest partial ovoid has been provided in \cite{BKMS1} and, if $q$ is even, a partial ovoid of an elliptic or hyperbolic quadric is also a partial ovoid of a symplectic polar space. 

As for $\cH(2n, q^2)$, an upper bound on the size of the largest partial ovoid can be found in \cite{BKMS}. In particular, a partial ovoid has at most $q^5-q^4+q^3+1$ points if $n = 2$. The largest known example of a maximal partial ovoid of $\cH(2n,q^2)$, $n = 2, 3$, occurs when $q=3^h$ and has size $q^4+1$ \cite{MPS}. A straightforward check shows that a non-degenerate plane section of $\cH(2n,q^2)$ is an example of maximal partial ovoid of $\cH(2n,q^2)$ of size $q^3+1$. Other examples of maximal partial ovoids of $\cH(4,q^2)$ of size $2q^3+q^2+1$ have been constructed in \cite{CP} and of size $q^3+1$ in \cite{CS, MPS}.

In Section~\ref{symp}, partial ovoids of symplectic polar spaces are studied. We show the existence of a partial ovoid of $\cW(3, q)$, $q$ an odd square, $q \not\equiv 0 \pmod{3}$, of size $(q^{3/2} +3q - q^{1/2}+3)/3$. It is obtained by glueing together a twisted cubic $\cC$ of $\PG(3, q)$ and an orbit of a subgroup of $\PSp(4, q)$ isomorphic to $\PGL(2, q^{1/2})$ stabilizing $\cC$. Next maximal partial ovoids of $\cW(5, q)$ of size $q^2+q+1$ and of size $2q^2-q+1$, if $q$ is even, are exhibited. 

In Section~\ref{Herm} we introduce the notion of {\em tangent-set} of a Hermitian variety $\cH(2n-1, q^2)$, that is a set $\cT$ of points of $\PG(2n-1, q^2)$ such that every line that is either tangent or contained in $\cH(2n-1, q^2)$ has at most one point in common with $\cT$. In particular we show that starting from a partial ovoid of $\cW(2n-1, q)$ of size $x$, it is possible to obtain a tangent-set of $\cH(2n-1, q^2)$ of size roughly $x q$, which in turn gives rise to a partial ovoid of $\cH(2n, q^2)$ of size roughly $x q^2$. Applying this construction technique to the known (partial) ovoids of $\cW(2n-1, q)$, $n \in \{2,3,4\}$, the following are obtained: maximal partial ovoids of $\cH(4, q^2)$, $q$ even, of size $q^4+1$ and of size $q^4-q^3+q+1$, a partial ovoid of $\cH(4, q^2)$, $q$ an odd square, $q \not\equiv 0 \pmod{3}$, of size $(q^{7/2}+3q^3-q^{5/2}+3q^2)/3$, partial ovoids of $\cH(6, q^2)$ of size $q^4+q^3+1$ and of size $2q^4-q^3+1$, if $q$ is even, and a partial ovoid of $\cH(8, q^2)$ of size $q^5+1$. Table~\ref{Tab1} and Table~\ref{Tab2} summarize old and new results regarding large partial ovoids of symplectic and Hermitian polar spaces in small dimensions.

\begin{table}[H]
$
\begin{tabular}{c||c|c}
 & lower bound & upper bound \\
\hline
\hline
$\cW(3, q)$, $q$ even & $q^2+1$ & $q^2+1$ \\
\hline
	& {\boldmath $(q^{3/2} +3q - q^{1/2}+3)/3$}, \\
$\cW(3, q)$, $q$ odd & $q = p^{2h}$, $p \ne 3$, & \\
	& & $q^2-q+1$ \cite{T} \\
	& $2q+1$, & \\
	& $q = p^{2h+1}$ or $q = 3^h$ \cite{T} & \\
\hline
	& {\boldmath $2q^2-q+1$}, $q$, even & \\
	& & $\frac{q\sqrt{5q^4+6q^3+7q^2+6q+1}-q^3-q^2-q+2}{2}$ \cite{BKMS1} \\
$\cW(5, q)$ & {\boldmath $q^2+q+1$}, $q$ odd & \\
	& & \\
	& $7$, $q = 2$ & $7$, $q = 2$ \\
\hline
	& & $q^4-q^3-q(\sqrt{q}-1)(q-\sqrt{q}+1)+3$, \\
$\cW(7, q)$ & $q^3+1$ \cite{C, CK} & $q > 2$ \cite{BKMS1} \\
	& & \\
	& & $9$, $q = 2$ \\  
\hline
\end{tabular}
$
\caption{\label{Tab1}Large partial ovoids of $\cW(2n-1, q)$, $n \in \{2, 3, 4\}$.}
\end{table}

\begin{table}[H]
$
\begin{tabular}{c||c|c}
 & lower bound & upper bound \\
\hline
\hline
	& {\boldmath $q^4+1$}, $q = 2^h$ or $q = 3^h$ \cite{MPS} & \\
	& & \\
$\cH(4, q^2)$ & {\boldmath $(q^{7/2}+3q^3-q^{5/2}+3q^2)/3$}, & $q^5-q^4+q^3+1$ \cite{BKMS} \\
	& $q = p^{2h}$, $p$ odd, $p \ne 3$ & \\
	& & \\
	& $2q^3+q^2+1$, & \\
	& $q = p^{2h+1}$, $p \ne 2, 3$ \cite{CP} & \\
\hline
	& {\boldmath $2q^4-q^3+1$}, $q$, even & \\
$\cH(6, q^2)$ & & $q^7-q^6+q^5-q^3+2$ \cite{BKMS}\\
	& {\boldmath $q^4+q^3+1$}, $q$ odd & \\
\hline
$\cH(8, q^2)$ & {\boldmath $q^5+1$} & $q^9-q^8+q^7-q^5-q^3+q^2+1$ \cite{BKMS}\\
\hline	
\end{tabular}
$
\caption{\label{Tab2}Large partial ovoids of $\cH(2n, q^2)$, $n \in \{2,3,4\}$.}
\end{table}

\section{Preliminaries}\label{pre}

Let $q$ be a prime power and let $\F_q$ be the finite field of order $q$. Let $\PG(n, q)$ be the $n$-dimensional projective space over $\F_{q}$ equipped with homogeneous projective coordinates $(X_1, X_2, \dots, X_{n+1})$. Denote by $U_i$ the point having $1$ in the $i$-th position and $0$ elsewhere. A non-degenerate symplectic or unitary polarity of $\PG(n, q^2)$ is induced by a non-degenerate alternating or Hermitian form on the underlying vector space. 

The symplectic polar space $\cW(2n-1, q)$ is formed by the projective subspaces that are absolute with respect to a non-degenerate symplectic polarity $\s$ of $\PG(2n-1, q)$. It is left invariant by the group of projectivities $\PSp(4,q)$. It consists of all the points of $\PG(2n-1, q)$ and of $(q+1) \dots (q^{n}+1)$ generators. Through every point $P$ of $\PG(2n-1, q)$ there pass $(q+1) \dots (q^{n-1}+1)$ generators and these generators all lie in a hyperplane. The hyperplane containing these generators is the polar hyperplane $P^\s$ of $P$ with respect to the symplectic polarity $\s$ defining $\cW(2n-1, q)$.

A (non-degenerate) Hermitian variety $\cH(n, q^2)$ consists of the absolute points of a unitary polarity of $\PG(n, q^2)$. It has $\left(q^{n+1}+(-1)^n\right)\left(q^n-(-1)^n\right)/\left(q^2-1\right)$ points. Subspaces of maximal dimension contained in $\cH(n, q^2)$ are $\lfloor \frac{n-1}{2} \rfloor$-dimensional projective spaces and are called {\em generators}. There are either $(q + 1)(q^3 + 1) \dots (q^n + 1)$ or $(q^3 + 1)(q^5 + 1) \dots (q^{n+1} + 1)$ generators in $\cH(n, q^2)$, depending on whether $n$ is odd or even, respectively.  A line of $\PG(n, q^2)$ meets $\cH(n, q^2)$ in $1$, $q+1$ or, if $n \ge 3$, in $q^2 + 1$ points. The latter lines are the generators of $\cH(n, q^2)$ if $n \in \{3, 4\}$; lines meeting $\cH(n, q^2)$ in one or $q + 1$ points are called {\em tangent lines} or {\em secant lines}, respectively. Through a point $P$ of $\cH(n, q^2)$ there pass $\left(q^{n-1}+(-1)^{n-2}\right)\left(q^{n-2}-(-1)^{n-2}\right)/\left(q^2-1\right)$ generators and these generators are contained in a hyperplane. The hyperplane containing these generators is the polar hyperplane of $P$ with respect to the unitary polarity of $\PG(n, q^2)$ defining $\cH(n, q^2)$ and it is also called the {\em tangent hyperplane} to $\cH(n, q^2)$ at $P$. The tangent lines through $P$ are precisely the remaining lines of its polar hyperplane that are incident with $P$. If $P \notin \cH(n,q^2)$ then the polar hyperplane of $P$ is a hyperplane of $\PG(n, q^2)$ meeting $\cH(n, q^2)$ in a non-degenerate Hermitian variety $\cH(n-1, q^2)$ and it is said to be {\em secant} to $\cH(n, q^2)$. The stabilizer of $\cH(n, q^2)$ in $\PGL(n+1, q^2)$ is the group $\PGU(n+1, q^2)$. Recall that the symplectic polar space $\cW(2n-1, q)$ can be embedded in $\cH(2n-1, q^2)$.

\section{Partial ovoids of symplectic polar spaces}\label{symp}

\subsection{${\mathcal W}(3, q)$, $q$ odd square, $q \not\equiv 0 \pmod{3}$}

Let $\cW(3, q)$, $q \not\equiv 0 \pmod{3}$, be the symplectic polar space consisting of the subspaces of $\PG(3, q)$ induced by the totally isotropic subspaces of $\F_{q}^4$ with respect to the non-degenerate alternating form $\beta$ given by 
\begin{align*}
x_1 y_4 + x_2 y_3 - x_3 y_2 - x_4 y_1. 
\end{align*} 
Denote by $\s$ the symplectic polarity of $\PG(3, q)$ defining $\cW(3, q)$. Let $\cC$ be the twisted cubic of $\PG(3, q)$ consisting of the $q+1$ points $\{P_t \mid t \in \F_q\} \cup \{(0,0,0,1)\}$, where $P_t = (1, -3 t, t^2, t^3)$. It is well known that a line of $\PG(3, q)$ meets $\cC$ in at most $2$ points and a plane shares with $\cC$ at most $3$ points (i.e., $\cC$ is a so called {\em $(q+1)$-arc}). A line of $\PG(3,q)$ joining two distinct points of $\cC$ is called a {\em real chord} and there are $q(q+1)/2$ of them. Let $\bar{\cC} = \{P_t \mid t \in \F_{q^2}\} \cup \{(0,0,0,1)\}$ be the twisted cubic of $\PG(3, q^2)$ which extends $\cC$ over $\F_{q^2}$. The line of $\PG(3, q^2)$ obtained by joining $P_t$ and $P_{t^q}$, with $t \notin \F_q$, meets the canonical Baer subgeometry $\PG(3, q)$ in the $q+1$ points of a line skew to $\cC$. Such a line is called an {\em imaginary chord} and they are $q(q-1)/2$ in number. If $r$ is a (real or imaginary) chord, then the line $r^{\s}$ is called ({\em real} or {\em imaginary}) {\em axis}. Also, for each point $P$ of $\cC$, the line $\ell_{P} = \langle P, P' \rangle$, where $P'$ equals $(0, -3, 2t, 3t^2)$ or $U_3$ if $P = P_t$ or $P = U_4$, respectively, is called the {\em tangent} line to $\cC$ at $P$. With each point $P_t$ (resp. $U_4$) of $\cC$ there corresponds the {\em osculating plane} $P_t^{\s}$ (resp. $U_4^{\s}$) with equation $t^3 X_1 + t^2 X_2 + 3t X_3 - X_4 = 0$ (resp. $X_1 = 0$), meeting $\cC$ only at $P_t$ (resp. $U_4$) and containing the tangent line. Hence the $q+1$ lines tangent  to $\cC$ are generators of $\cW(3, q)$ and they form a regulus $\cR$ if $q$ even. Every point of $\PG(3, q) \setminus \cC$ lies on exactly one chord or a tangent of $\cC$. For more properties and results on $\cC$ the reader is referred to \cite[Chapter 21]{H2}. Let $G$ be the group of projectivities of $\PG(3, q)$ stabilizing $\cC$. Then $G \simeq \PGL(2, q)$ whenever $q \ge 5$, and elements of $G$ are induced by the matrices
\begin{align*} 
M_{a,b,c,d} = \begin{pmatrix}
a^3 & -a^2 b & 3 a b^2 & b^3 \\
-3a^2 c & a^2 d + 2 abc & -3b^2 c -6 abd & -3b^2 d \\
a c^2 & \frac{-bc^2 - 2 acd}{3} & ad^2 + 2 bcd & b d^2 \\
c^3 & - c^2 d & 3 c d^2 & d^3 \\
\end{pmatrix},
\end{align*}
where $a,b,c,d \in \F_q$, $ad-bc \ne 0$. The group $G$ leaves invariant $\cW(3, q)$ since $M_{a,b,c,d}^t J M_{a,b,c,d} = (ad-bc)^3 J$, where   
\begin{align*}
& J = \begin{pmatrix}
0 & 0 & 0 & 1 \\
0 & 0 & 1 & 0 \\
0 & -1 & 0 & 0 \\
-1 & 0 & 0 & 0 \\
\end{pmatrix}
\end{align*} 
is the Gram matrix of $\beta$. 

{\bf Assume $q$ to be an odd square.} Here we show the existence of a partial ovoid of $\cW(3, q)$ obtained by glueing together the twisted cubic $\cC$ of $\PG(3, q)$ and an orbit of size $\sqrt{q}(q-1)/3$ of a subgroup $G_{\epsilon}$ of $\PSp(4, q) \cap G$ isomorphic to $\PGL(2, \sqrt{q})$ stabilizing $\cC$. In particular such a subgroup fixes a twisted cubic $\cC_{\epsilon} = \cC \cap \Sigma_{\epsilon}$ of a Baer subgeometry $\Lambda_{\epsilon}$ of $\PG(3, q)$. 

\fbox{$\sqrt{q} \equiv 1 \pmod{3}$}

Assume $\sqrt{q} \equiv 1 \pmod{3}$. Let $\cC_{1} = \{P_t \mid t \in \F_{\sqrt{q}} \} \cup \{U_4\}$. Thus $\cC_1 \subset \cC$ is a twisted cubic of the canonical Baer subgeometry $\Lambda_{1} = \PG(3, \sqrt{q})$ of $\PG(3, q)$. The group $G_1$ of projectivities stabilizing $\cC_1$ is isomorphic to $\PGL(2, \sqrt{q})$ and it is induced by the matrices
\begin{align*}
& M_{a,b,c,d}, a,b,c,d \in \F_{\sqrt{q}}, ad-bc \ne 0.
\end{align*}
Let $R = U_1 + x U_4$, for a fixed $x \in \F_{q} \setminus \F_{\sqrt{q}}$ where $x$ is not a cube in $\F_{q}$. Set 
\begin{align*}
& \cO_1 = R^{G_1} = \{(a^3+xb^3, -3a^2c-3xb^2d, ac^2+xbd^2, c^3+xd^3) \mid a,b,c,d \in \F_{\sqrt{q}}, ad-bc \ne 0\}. 
\end{align*}
\begin{lemma}
The set $\cO_1$ is a partial ovoid of $\cW(3, q)$ of size $\sqrt{q}(q-1)/3$.
\end{lemma}
\begin{proof}
Since the stabilizer of $R$ in $G_1$ has order $3$ and $|G_1| = \sqrt{q}(q-1)$, by applying the Orbit-Stabilizer Theorem it follows that $\cO_1$ has the required size. In order to show that $\cO_1$ is a partial ovoid of $\cW(3, q)$, it is enough to see that the line joining $R$ and a further point $R^g$ of $\cO_1$ is not a generator of $\cW(3, q)$. Here $g \in G_1$ is induced by $M_{a,b,c,d}$. Assume by contradiction that this is not the case, then there are $a,b,c,d \in \F_{\sqrt{q}}$, with $ad-bc \ne 0$, such that $A(x) = 0$, where
\begin{align*}
& A(x) = c^3 + x d^3 -x a^3 -x^2 b^3.
\end{align*}
Hence $A(x) + A(x)^{\sqrt{q}} = A(x) - A(x)^{\sqrt{q}} = 0$, that is
\begin{align*}
& 2c^3 + (d^3-a^3) (x+x^{\sqrt{q}}) - (x^2+x^{2\sqrt{q}}) b^3 = 0, \\
& d^3 - a^3 -(x+x^{\sqrt{q}}) b^3 = 0.
\end{align*}  
If $b = 0$, then $a^3 = d^3$ and $c = 0$. Thus $g$ fixes $R$, i.e., $R = R^g$. If $b \ne 0$, then the previous equations imply
\begin{align*}
x^{\sqrt{q}+1} = - \frac{c^3}{b^3},
\end{align*}    
that is a contradiction since $x^{\sqrt{q}+1}$ is not a cube in $\F_{\sqrt{q}}$. Indeed $x$ is not a cube in $\F_{q}$ and $(\sqrt{q}+1, 3) = 1$.
\end{proof}

\fbox{$\sqrt{q} \equiv -1 \pmod{3}$} 

Assume $\sqrt{q} \equiv -1 \pmod{3}$. Let $\cC_{-1} = \{P_t \mid t \in \F_{q}, t^{\sqrt{q}+1} = 1 \}$. Thus $\cC_{-1} \subset \cC$ is a twisted cubic of the Baer subgeometry 
\begin{align*}
& \Lambda_{-1} = \{(\alpha, -3\beta, \beta^q, \alpha^q) \mid \alpha, \beta \in \F_{q^2}, (\alpha, \beta) \ne (0, 0)\} \simeq \PG(3, \sqrt{q})
\end{align*}
of $\PG(3, q)$. In this case the group $G_{-1}$ of projectivities stabilizing $\cC_{-1}$ is isomorphic to $\PGL(2, \sqrt{q})$ and it is induced by the matrices
\begin{align*}
& M_{a,b,c,d}, a,b,c,d \in \F_{q}, ad-bc \ne 0, ab^{\sqrt{q}} - cd^{\sqrt{q}} = 0, a^{\sqrt{q}+1} + b^{\sqrt{q}+1} - c^{\sqrt{q}+1} - d^{\sqrt{q}+1} = 0.
\end{align*}
Let $S = U_1 + x U_4$, for a fixed $x \in \F_{q} \setminus \F_{\sqrt{q}}$ where $x$ is not a cube in $\F_{q}$ and $x^{\sqrt{q}+1} \ne 1$. Set 
\begin{align*}
& & \cO_{-1} = S^{G_{-1}} = \{(a^3+xb^3, -3a^2c-3xb^2d, ac^2+xbd^2, c^3+xd^3) \mid a,b,c,d \in \F_{q}, ad-bc \ne 0, \\
& & ab^{\sqrt{q}} - cd^{\sqrt{q}} = 0, a^{\sqrt{q}+1} + b^{\sqrt{q}+1} - c^{\sqrt{q}+1} - d^{\sqrt{q}+1} = 0 \}. 
\end{align*}
\begin{lemma}
The set $\cO_{-1}$ is a partial ovoid of $\cW(3, q)$ of size $\sqrt{q}(q-1)/3$.
\end{lemma}
\begin{proof}
Since the stabilizer of $S$ in $G_{-1}$ has order $3$ and $|G_{-1}| = \sqrt{q}(q-1)$, by applying the Orbit-Stabilizer Theorem it follows that $\cO_{-1}$ has the required size. In order to show that $\cO_{-1}$ is a partial ovoid of $\cW(3, q)$, it is enough to see that the line joining $R$ and a further point $S^g$ of $\cO_{-1}$ is not a generator of $\cW(3, q)$. Here $g \in G_{-1}$ is induced by $M_{a,b,c,d}$. Assume by contradiction that this is not the case, then there are $a,b,c,d \in \F_{q}$, with $ad-bc \ne 0$, $ab^{\sqrt{q}} - cd^{\sqrt{q}} = 0$, $a^{\sqrt{q}+1} + b^{\sqrt{q}+1} - c^{\sqrt{q}+1} - d^{\sqrt{q}+1} = 0$, such that $A(x) = 0$, where
\begin{align*}
& A(x) = c^3 + x d^3 -x a^3 -x^2 b^3.
\end{align*}
If $b = 0$, then $c = 0$, since $cd^{\sqrt{q}} = 0$ and $ad \ne 0$. Moreover $a^{\sqrt{q}+1} = d^{\sqrt{q}+1}$ and $a^3 = d^3$, since $A(x) = 0$. Thus $g$ fixes $S$, i.e., $S = S^g$. If $b \ne 0$, we may assume w. l. o. g. that $b = 1$. Then $a = cd^{\sqrt{q}}$, $c(d^{\sqrt{q}+1}-1) \ne 0$ and $(1-c^{\sqrt{q}+1})(1-d^{\sqrt{q}+1}) = 0$. Therefore $c^{\sqrt{q}+1} = 1$ and $d^{\sqrt{q}+1} \ne 1$. Moreover $A(x) = 0$ gives
\begin{align*}
& d^{3\sqrt{q}} - \frac{d^3}{c^3} - \frac{c^3 -x^2}{c^3 x} = 0. 
\end{align*}
By considering the last equation in the unknown $d^3$, by \cite[Theorem 1.9.3]{H1}, it admits solutions in $\F_{q}$ if and only 
\begin{align*}
& 0  = \frac{c^{3\sqrt{q}} - x^{2\sqrt{q}}}{c^{3\sqrt{q}} x^{\sqrt{q}}} + \frac{c^3 -x^2}{c^3 x} \frac{1}{c^{3\sqrt{q}}} = \frac{x (1-x^{\sqrt{q}+1}) (c^{3\sqrt{q}} + x^{\sqrt{q}-1})}{c^{3\sqrt{q}} x^{\sqrt{q}+1}}. 
\end{align*} 
Hence either $x = 0$ or $x^{\sqrt{q}+1} = 1$ or $x^{\sqrt{q}-1}$ is a cube in $\F_{q}$. If the last possibility occurs then $x$ is a cube in $\F_{q}$ since $\sqrt{q} \equiv -1 \mod{3}$. We infer that none of the three cases arises. 
\end{proof}

\begin{theorem}\label{partial-symp3}
Let $q$ be an odd square with $\sqrt{q} \equiv \epsilon \pmod{3}$, $\epsilon \in \{\pm1\}$. Then the set $\cO_{\epsilon} \cup \cC$ is a partial ovoid of $\cW(3, q)$, of size $(q^{3/2} +3q - q^{1/2} +3)/3$.
\end{theorem}
\begin{proof}
Since in both cases $x$ is not a cube in $\F_q$, it follows that there is no line of $\cW(3, q)$ through $R$ or $S$ meeting $\cC$. 
\end{proof}

\begin{remark}\label{non-complete}
The partial ovoid $\cO_{\epsilon} \cup \cC$ of $\cW(3, q)$ is not maximal. Indeed if $N$ is a point of $\Lambda_{\epsilon}$ not lying on an osculating plane of $\cC_{\epsilon}$, then $\cO_{\epsilon} \cup \cC \cup \{N\}$ is a partial ovoid of $\cW(3, q)$. In order to see this fact let $\ell$ be a line of $\cW(3, q)$. If $\ell$ passes through a point of  $\cC_{\epsilon}$ then it lies in an osculating plane of $\cC_{\epsilon}$, whereas if $\ell$ contains a point of $\cC \setminus \Lambda_{\epsilon}$, then $\ell \cap \Lambda_{\epsilon}$ is a point lying on an imaginary axis of $\cC_{\epsilon}$. Similarly, since the points of $\cO_{\epsilon}$ are on extended chords of $\cC_{\epsilon}$, if $\ell$ contains a point of $\cO_{\epsilon}$, then $\ell \cap \Lambda_{\epsilon}$ is a point lying on an axis of $\cC_{\epsilon}$. On the other hand every axis of $\cC_{\epsilon}$ is contained in an osculating plane of $\cC_{\epsilon}$. 
\end{remark}

\begin{prob}
Find a maximal partial ovoid of $\cW(3, q)$ containing $\cO_{\epsilon} \cup \cC$.
\end{prob}

\subsection{Maximal partial ovoids of ${\mathcal W}(5, q)$ of size $q^2+q+1$}

Here we consider a symplectic polar space in $\PG(5, q)$, where $\PG(5, q)$ is embedded as a subgeometry in $\PG(5, q^3)$. We denote by $N: x \in \F_{q^3} \mapsto x^{q^2+q+1} \in \F_q$ the {\em norm function} of $\F_{q^3}$ over $\F_q$ and by $T: x \mapsto x+x^q+x^{q^2}$ the {\em trace function} of $\F_{q^3}$ over $\F_q$. Let $V$ be the $6$-dimensional $\F_q$-vector subspace of $\F_{q^3}^6$ given by 
\begin{align*}
&\{P_{a,b} = (a,a^q, a^{q^2}, b^{q^2}, b^q, b) \mid a, b \in \F_{q^3}\}.
\end{align*}
Then $\PG(V)$ is a $q$-order subgeometry of $\PG(5, q^3)$. If $\cW(5, q^3)$ is the symplectic polar space consisting of the subspaces of $\PG(5, q^3)$ induced by the totally isotropic subspaces of $\F_{q^3}^6$ with respect to the non-degenerate alternating form given by 
\begin{align*}
x_1 y_6 + x_2 y_5 + x_3 y_4 - x_4 y_3 - x_5 y_2 - x_6 y_1, 
\end{align*}
then $\cW(5, q^3)$ induces in $\PG(V)$ a symplectic polar space, say $\cW(5, q)$. It is straightforward to check that the cyclic group of order $q^2+q+1$ formed by the projectivities of $\PG(V)$ induced by the matrices
\begin{align*}
& D_x = \begin{pmatrix}
x & 0 & 0 & 0 & 0 & 0 \\
0 & x^q & 0 & 0 & 0 & 0 \\
0 & 0 & x^{q^2} & 0 & 0 & 0\\
0 & 0 & 0 & x^{-q^2} & 0 & 0 \\
0 & 0 & 0 & 0 & x^{-q} & 0 \\
0 & 0 & 0 & 0 & 0 & x^{-1} \\
\end{pmatrix}, x \in \F_{q^3}, N(x) = 1,
\end{align*}
preserves $\cW(5, q)$. The group $K$ fixes the two planes
\begin{align*}
& \pi_1 = \{P_{a, 0} \mid a \in \F_{q^3} \setminus \{0\}\}, 
& \pi_2 = \{P_{0, b} \mid b \in \F_{q^3} \setminus \{0\}\}.
\end{align*}

\begin{lemma}
There are one or three $K$-orbits on points of $\pi_i$, $i = 1,2$, according as $q \not\equiv 1 \pmod{3}$ or $q \equiv 1 \pmod{3}$. In the latter case their representatives are as follows 
\begin{align*}
& \pi_1: P_{z^i,0}, 
\pi_2: P_{0, z^i}, i = 1,2,3, 
\end{align*}
for some $z \in \F_{q^3}$ such that $z^{q-1} = \xi$, where $\xi$ is a fixed element in $\F_q$ such that $\xi^2+\xi+1 = 0$. 

The group $K$ has  $q^3-1$ orbits on points of $\PG(V) \setminus (\pi_1 \cup \pi_2)$. Each of them has size $q^2+q+1$ and their representatives are the points 
\begin{align*}
 P_{1, b}, \; b \in \F_{q^3} \setminus \{0\}, & \quad q \not\equiv 1 \pmod{3}, \\
 P_{z^i, b^3}, i = 1,2,3, \; b \in \F_{q^3} \setminus \{0\}, & \quad q \equiv 1 \pmod{3}.
\end{align*}
\end{lemma}
\begin{proof}
The element of $K$ induced by $D_x$ fixes a point of $\pi_1$ (or of $\pi_2$) if and only if $x^{q-1} = x^{q^2+q+1} = 1$, that is $x^3 = 1$, with $x \in \F_q$. Therefore if $q \not\equiv 1 \pmod{3}$, then $K$ permutes in a single orbit the points of $\pi_1$ (or of $\pi_2$). If $q \equiv 1 \pmod{3}$, then the kernel of the action of $K$ on both $\pi_1$ and $\pi_2$ consists of the subgroup of order three induced by $\langle D_\xi \rangle$. Let $z \in \F_{q^3}$ be such that $z^{q-1} = \xi$,  where $\xi$ is a fixed element in $\F_q$ such that $\xi^2+\xi+1 = 0$. In order to see that the representatives of the $K$-orbits on points of $\pi_1$ (or of $\pi_2$) are $P_{z^i, 0}$ (or $P_{0, z^i}$), $i = 1,2,3$, observe that 
\begin{align*}
& \F_{q^3} \setminus \{0\}= \bigcup_{i = 1}^3 \{x z^i \mid x \in \F_{q^3}, x^{q^2+q+1} = 1\}.
\end{align*}

Let $R$ be a point of $\PG(V) \setminus (\pi_1 \cup \pi_2)$. It is straightforward to check that no non-trivial element of $K$ leaves $R$ invariant, that is $|R^K| = q^2+q+1$. Since the planes $\pi_1$, $\pi_2$ are disjoint, there exists a unique line $\ell$ of $\PG(V)$ intersecting both $\pi_1$ and $\pi_2$ and containing $R$. Hence if $|\ell \cap R^K| \ge 2$, then there exists a non-trivial element in $K$ fixing $\ell$ and hence stabilizing both $\ell \cap \pi_1$ and $\ell \cap \pi_2$. From the discussion above such a non-trivial element exists if and only if $q \equiv 1 \pmod{3}$. If $q \not\equiv 1 \pmod{3}$, it follows that the representatives of the $K$-orbits on points of $\PG(V) \setminus (\pi_1 \cup \pi_2)$ can be taken as the points on the $q^2+q+1$ lines through $P_{1,0}$ and meeting $\pi_2$ in a point. Similarly, if $q \equiv 1 \pmod{3}$, then these representatives can be chosen among the points of $\PG(V) \setminus (\pi_1 \cup \pi_2)$ lying on the $3(q^2+q+1)$ lines through $P_{z^i,0}$, $i = 1,2,3$, and meeting $\pi_2$ in a point. In particular, in this case, $P_{z^i, b}$ and $P_{z^i, b'}$ lie in the same $K$-orbit if $b' = \xi^i b$, $i = 1,2,3$. The result now follows.  
\end{proof}

\begin{theorem}\label{partial-symp5}
If $c \in \F_q \setminus \{0\}$, the orbit $\cO = P_{1, c}^K$ is a partial ovoid of $\cW(5, q)$ of size $q^2+q+1$.
\end{theorem}
\begin{proof}
Let $P'$ be a point of $P_{1, c}^K \setminus \{P_{1, c}\}$. Then there exists $x \in \F_{q^3}$, with $N(x) = 1$, $x \ne 1$, such that 
\begin{align*}
& P' = (x, x^q, x^{q^2}, c x^{-q^2}, c x^{-q}, c x^{-1}).
\end{align*}
If the line joining $P_{1, c}$ with $P'$ is a line of $\cW(5, q)$, then
\begin{align*}
0 & = c T(x^{-1} - x)  = c T(x^{q+1} - x) = c N(1-x) 
\end{align*}
Therefore $x = 1$ and $P_{1, c} = P'$, a contradiction.
\end{proof}

In order to prove that $\cO$ is maximal we need to recall some results obtained by Culbert and Ebert in \cite{CE} in terms of Sherk surfaces. A {\em Sherk surface} $S(\alpha, \beta, \gamma, \delta)$, where $\alpha, \delta \in \F_{q}$, $\beta, \gamma \in \F_{q^3}$, can be seen as a hypersurface of the projective line $\PG(1, q^3)$. More precisely
\begin{align*}
S(\alpha, \beta, \gamma, \delta) = \{x \in \F_{q^3} \cup \{\infty\} \mid \alpha N(x) + T(\beta^{q^2} x^{q+1}) + T(\gamma x) + \delta = 0\}.
\end{align*}

\begin{lemma}[\cite{CE}]\label{ebert}
Let $\alpha, \delta \in \F_{q}$, $\beta, \gamma \in \F_{q^3}$.
\begin{itemize}
\item[1)] $|S(\alpha, \beta, \gamma, \delta)| \in \{1, q^2-q+1, q^2+1, q^2+q+1\}$.
\item[2)] $\infty \in S(\alpha, \beta, \gamma, \delta)$ if and only if $\alpha = 0$.
\item[3)] $|S(\alpha, \beta, \gamma, \delta)| = 1$ if and only if $(\alpha, \beta, \gamma, \delta) \in \{(1, \beta, \beta^{q^2+q}, N(\beta)), (0,0,0,1)\}$.
\end{itemize}
\end{lemma}

Note that $\lambda S(\alpha_1, \beta_1, \gamma_1, \delta_1) + \mu S(\alpha_2, \beta_2, \gamma_2, \delta_2) = S(\lambda \alpha_1 + \mu \alpha_2, \lambda \beta_1 + \mu \beta_2, \lambda \gamma_1 + \mu \beta_2, \lambda \delta_1 + \mu \delta_2)$ for any $\lambda, \mu \in \F_q$, $(\lambda, \mu) \ne (0, 0)$. Therefore any two distinct Sherk surfaces give rise to a pencil containing $q+1$ distinct Sherk surfaces which together contain all the points of $\PG(1, q^3)$. Moreover, any two distinct Sherk surfaces in a given pencil intersect in the same set of points, called the {\em base locus} of the pencil, which is the intersection of all the Sherk surfaces in that pencil.

\begin{lemma}\label{sherk}
Let $\beta, \gamma \in \F_{q^3}$ not both zero. If one of the following is satisfied 
\begin{itemize}
\item $\beta \gamma = 0$,
\item $N(\beta) \ne c^3$, for all $c \in \F_q \setminus \{0\}$, 
\item $N(\beta) = c^3$, for some $c \in \F_q \setminus \{0\}$, and $\gamma \beta \ne -c^2$,
\end{itemize}
then the base locus of the pencil generated by $S(1,0,0,-1)$ and $S(0,\beta, \gamma,0)$ is not empty.
\end{lemma}
\begin{proof}
The pencil consists of $S(0, \beta, \gamma, 0)$ and $S(1,0,0,-1) + \lambda S(0, \beta, \gamma, 0) = S(1, \lambda \beta, \lambda \gamma, -1)$, where $\lambda \in \F_q$. Let $y_i$ be the number of Sherk surfaces in the pencil having $i$ points. Hence $y_i$ are non-negative integer and $\t \ge 1$ since $|S(1,0,0,-1)| = q^2+q+1$. By Lemma \ref{ebert}, it follows that $|S(0, \beta, \gamma, 0)| \ne 1$ and $|S(1, \lambda \beta, \lambda \gamma, -1)| \ne 1$, for $\lambda \in \F_q$, i.e., $y_1 = 0$. Assume by contradiction that the base locus of the pencil is empty, then the following equations hold
\begin{align*}
& \y + \z + \t = q+1, \\
& (q^2-q+1) \y + (q^2+1) \z + (q^2+q+1) \t = q^3+1.
\end{align*} 
By substituting $\y = q+1 - \z - \t$ in the second equation, we have that 
\begin{align*}
& \t = - \frac{\z}{2}. 
\end{align*}
Therefore necessarily 
\begin{align*}
& \t = \z = 0, \y = q+1,
\end{align*} 
which is a contradiction. 
\end{proof}

\begin{theorem}
The partial ovoid $\cO$ of $\cW(5, q)$ is maximal.
\end{theorem}
\begin{proof}
It is enough to show that for each of the representatives $R$ of the $K$-orbits on points of $\PG(V) \setminus \cO$, there exists a point $P = (x, x^q, x^{q^2}, c x^{-q^2}, c x^{-q}, c x^{-1}) \in \cO$ such that the line joining $R$ and $P$ is a line of $\cW(5, q)$. If $R$ is a point of $\pi_1$, then $R = P_{1,0}$ if $q \not\equiv 1 \pmod{3}$ or $R \in \{P_{z^i,0} \mid i = 1,2,3\}$, if $q \equiv 1 \pmod{3}$. The line $R P$ belongs to $\cW(5, q)$ if and only if there exists $x \in \F_{q^3}$, with $N(x) = 1$ such that 
\begin{align}
& c T(u x^{-1}) = 0, \label{conj1}
\end{align}  
where $u = 1$ or $u \in \{z^i \mid i = 1,2,3\}$ according as $q \not \equiv 1 \pmod{3}$ or $q \equiv 1 \pmod{3}$. Since $c \ne 0$ and $N(x) = 1$, \eqref{conj1} is equivalent to $T(u^{q^2} x^{q+1}) = 0$, i.e., $x \in S(0, u, 0, 0)$. Hence the existence of a line $R P$ of $\cW(5, q)$ is equivalent to the following  
\begin{align}
& |S(1,0,0,-1) \cap S(0, u, 0, 0)| \ge 1. \label{sherk1}
\end{align}  
As before let $u = 1$ or $u \in \{z^i \mid i = 1,2,3\}$ according as $q \not \equiv 1 \pmod{3}$ or $q \equiv 1 \pmod{3}$. Let $R$ be the point $P_{0, u} \in \pi_2$ or $P_{u, b} \in \PG(V) \setminus (\pi_1 \cup \pi_2 \cup \cO)$, where $b \in \F_{q^3} \setminus \{0\}$, with $b^3 \ne c^3$, if $u = 1$, otherwise $P_{1, b} \in \cO$. Arguing in a similar way we get that the existence of a line $R P$ of $\cW(5, q)$ is equivalent to
\begin{align}
& |S(1,0,0,-1) \cap S(0, 0, u, 0)| \ge 1 \mbox{ or }  |S(1,0,0,-1) \cap S(0, cu, -b, 0)| \ge 1.  \label{sherk2}
\end{align}   
Since $N(u)$ is a cube in $\F_q$ if and only if $u = 1$ and if $u = 1$, then $b \ne c$, the inequalities in \eqref{sherk1} and \eqref{sherk2} are satisfied by Lemma~\ref{sherk}.
\end{proof}

\subsection{Maximal partial ovoids of $\cW(5, q)$, $q$ even, of size $2q^2-q+1$}

Assume $q$ to be even and let $\cW(5, q)$ be the symplectic polar space of $\PG(5, q)$ induced by the totally isotropic subspaces of $\F_{q}^6$ with respect to the non-degenerate alternating form given by 
\begin{align*}
x_1 y_2 + x_2 y_1 + x_3 y_4 + x_4 y_3 + x_5 y_6 + x_6 y_5. 
\end{align*}
Recall that $\s$ is the symplectic polarity of $\PG(5, q)$ associated with $\cW(5, q)$. Fix $\delta_i \in \F_q$ such that the polynomial $X^2+X+\delta_i$ is irreducible over $\F_q$, $i = 1,2$. Let $\Delta_1$ and $\Delta_2$ be the three-spaces given by $X_5 = X_6 = 0$ and $X_3 = X_4 = 0$, respectively; consider the three-dimensional elliptic quadrics $\cE_i \subset \Delta_i$ defined as follows 
\begin{align*}
& \cE_1 : X_1 X_2 + X_3^2+X_3X_4 + \delta_1 X_4^2 = 0, \\
& \cE_2 : X_1 X_2 + X_5^2+X_5X_6 + \delta_2 X_5^2 = 0.
\end{align*}
Denote by $\sigma$ the plane of $\Delta_1$ spanned by $U_1, U_2, U_3$; hence $\cE_1 \cap \sigma$ is a conic and $U_3^\s \cap \Delta_1 = \sigma$. In the hyperplane $U_3^\s: X_4 = 0$, let $\cP_1$ be the cone having as vertex the point $U_3$ and as base $\cE_2$ and let $\cP_2$ be the cone having as vertex the line $\Delta_1^\s$ and as base the conic $\cE_1 \cap \sigma$. Set 
\begin{align*}
& \cA = \cP_1 \cap \cP_2 = \{(1, c^2+cd+\delta_2d^2, \sqrt{c^2+cd+\delta_2d^2}, 0, c,d) \mid c, d \in \F_q\} \cup \{U_2\}.
\end{align*}

\begin{lemma}
The point sets $\cE_1$ and $\cA$ are partial ovoids of $\cW(5, q)$ of size $q^2+1$.
\end{lemma}
\begin{proof}
The elliptic quadric $\cE_1$ is an ovoid of $\Delta_1 \cap \cW(5, q)$. Hence it is a partial ovoid of $\cW(5, q)$. The set $\cA$ is contained in $U_3^\s$ and each of the lines obtained by joining $U_3$ with a point of $\cA$ intersects $\Delta_2$ in a point of $\cE_2$. In particular this gives a bijective correspondence between the points of $\cA$ and those of $\cE_2$. In order to see that $\cA$ is a partial ovoid of $\cW(5, q)$ note that the elliptic quadric $\cE_2$ is an ovoid of $\Delta_2 \cap \cW(5, q)$. 
\end{proof}

\begin{theorem}\label{partial-symp5-even}
The point set $\cA \cup (\cE_1 \setminus \sigma)$ is a maximal partial ovoids of $\cW(5, q)$ of size $2q^2-q+1$.
\end{theorem}
\begin{proof}
Taking into account the previous lemma, it is enough to show that if $P$ is a point of $\cE_1 \setminus \sigma$, then $|P^\s \cap \cA| = 0$. By contradiction let $P = (1, a^2+ab+\delta_1b^2, a, b, 0, 0) \in \cE_1 \setminus \sigma$ and $R = (1, c^2+cd+\delta_2d^2, \sqrt{c^2+cd+\delta_2d^2}, 0, c, d) \in \cA$, for some $a, b, c, d \in \F_q$, $b \ne 0$, with $R \in P^\s$. Then $\sqrt{c^2+cd+\delta_2d^2} \in \F_q$ is a root of $X^2+bX+a^2+ab+\delta_1 b^2$, which is irreducible over $\F_q$, a contradiction.   

In order to prove the maximality, let $Q$ be a point not in $\cA \cup (\cE_1 \setminus \sigma)$ and assume that $|Q^\s \cap (\cE_1 \setminus \sigma)| = 0$. There are two possibilities: either $Q^\s \cap \Delta_1$ coincides with $\sigma$ or $Q^\s \cap \Delta_1$ is a plane meeting $\cE_1$ in exactly one point of the conic $\cE_1 \cap \sigma$. In the former case $Q^\s$ meets $\cA$ in the points $\{U_1, U_2\}$. If the latter possibility occurs, let $R$ be the point of the conic $\cE_1 \cap \sigma$  contained in $Q^\s \cap \Delta_1$; since $R^\s \cap \Delta_1 \subset Q^\s$, the point $Q$ lies in the plane spanned by $R$ and $\Delta_1^\s$ . Hence the line joining $R$ and $Q$ is a line of $\cW(5, q)$ and contains a unique point of $\cA$.
\end{proof}

\begin{remark}
Some computations performed with Magma \cite{magma} show that the largest partial ovoid of $\cW(5, q)$, has size $7$ for $q = 2$ and $13$ for $q = 3$. In particular, $\cW(5, 2)$ has a unique partial ovoid of size $7$, up to projectivities, which coincides with the partial ovoids described above.
\end{remark}

\section{Partial ovoids of Hermitian polar spaces}\label{Herm}

Let $\cH(2n, q^2)$, $n \ge 2$, be the Hermitian polar space of $\PG(2n, q^2)$ and let $\perp$ be its associated unitary polarity. Fix a point $P \notin \cH(2n, q^2)$ and a pointset $\cT \subset P^\perp$ suitably chosen. In this section we explore the partial ovoids of $\cH(2n, q^2)$ that are constructed by taking the points of $\cH(2n, q^2)$ on the lines obtained by joining $P$ with the points of $\cT$. It turns out that in order to get a partial ovoid, the set $\cT$ has to intersect the lines that are tangent or contained in $P^\perp \cap \cH(2n, q^2)$ in at most one point.

\subsection{Tangent-sets of $\cH(2n-1, q^2)$, $n \in \{2, 3, 4\}$}

Let $\cH(2n-1, q^2)$ be a non-degenerate Hermitian variey of $\PG(2n-1, q^2)$. We introduce the notion of {\em tangent set} of $\cH(2n-1, q^2)$.
\begin{defin}
A {\bf tangent-set} of a Hermitian variety $\cH(2n-1, q^2)$ is a set $\cT$ of points of $\PG(2n-1, q^2)$ such that every line that is either tangent or contained in $\cH(2n-1, q^2)$ has at most one point in common with $\cT$. A tangent-set is said to be {\em maximal} if it is not contained in a larger tangent-set.  
\end{defin}
Here the main aim is to show that starting from a (partial) ovoid of $\cW(2n-1, q)$ it is possible to obtain a tangent set of $\cH(2n-1, q^2)$. In order to do that some preliminary results are needed. 

\begin{lemma}\label{ovoid}
An ovoid of $\cW(3, q) \subset \cH(3, q^2)$ is a maximal partial ovoid of $\cH(3, q^2)$. 
\end{lemma}
\begin{proof}
Let $\cX$ be an ovoid of $\cW(3, q) \subset \cH(3, q^2)$. Then obviously $\cX$ is a partial ovoid of $\cH(3, q^2)$. In order to see that $\cX$ is maximal as a partial ovoid of $\cH(3, q^2)$, observe that every point of $\cH(3, q^2) \setminus \cW(3, q)$ lies on a unique extended line of $\cW(3, q)$. 
\end{proof}

Let us fix an element $\i \in \F_{q^2}$ such that $\i + \i^q = 0$ and let $\xi_1, \dots, \xi_q$ be $q$ elements of $\F_{q^2}$ such that 
\begin{align}
& \left\{\i(\xi_i^q - \xi_i) \mid i = 1, \dots, q\right\} = \F_q.  \label{all}
\end{align}
In particular we may assume $\xi_1 = 0$. Let $\cH_i$ be the Hermitian variety given by
\begin{align*}
& X_1 X_{n+1}^q - X_1^q X_{n+1} + \dots + X_{n} X_{2n}^q - X_{n}^q X_{2n} + (\xi_i^q-\xi_i) X_{2n}^{q+1} = 0, \; i = 1, \dots, q.
\end{align*}
Then, if $\Pi$ denotes the hyperplane $X_{2n} = 0$, the set $\{\cH_i \mid i = 1, \dots, q\} \cup \{\Pi\}$ is a pencil of Hermitian varieties whose base locus consists of $\Pi \cap \cH_i$. Let $\Sigma_i$ denote the Baer subgeometry of $\PG(2n-1, q^2)$ consisting of the points
\begin{align*}
 \left\{(a_1, \dots, a_{n-1}, a_n + \xi_i, a_{n+1}, \dots, a_{2n-1}, 1) \mid a_1, \dots, a_{2n-1} \in \F_q, (a_1, \dots, a_{2n-1}) \ne (0, \dots, 0) \right\} & \\
 \cup \left\{(a_1, \dots, a_{2n-1}, 0) \mid a_1, \dots, a_{2n-1} \in \F_q, (a_1, \dots, a_{2n-1}) \ne (0, \dots, 0) \right\} &,
\end{align*}
for  $i = 1, \dots, q$. It is easily seen that $\Sigma_i \subset \cH_i$, for  $i = 1, \dots, q$. Hence $|\cH_1 \cap \left(\Sigma_i \setminus \Pi\right)| = 0$, if $i \ne 1$. Moreover, the lines of $\cH_i$ having $q+1$ points in common with $\Sigma_i$ are those of a symplectic polar space $\cW_i$ of $\Sigma_i$. Let $K$ be the group of projectivities of $\PG(2n-1, q^2)$ of order $q^{2n-1}$ induced by the matrices
\begin{align*}
\begin{pmatrix}
1 & 0 & \dots & 0 & 0 & 0 & \dots & 0 & a_1 \\
0 & 1 & \dots & 0 & 0 & 0 & \dots & 0 & a_2 \\
\vdots & \vdots & \ddots & \vdots & \vdots & \vdots & \ddots & \vdots & \vdots \\
0 & 0 & \dots & 1 & 0 & 0 & \dots & 0 & a_{n-1} \\
-a_{n+1} & -a_{n+2} & \dots & -a_{2n-1} & 1 & a_1 & \dots & a_{n-1} & a_{n} \\
0 & 0 & \dots & 0 & 0 & 1 & \dots & 0 & a_{n+1} \\
\vdots & \vdots & \ddots & \vdots & \vdots & \vdots & \ddots & \vdots & \vdots \\
0 & 0 & \dots & 0 & 0 & 0 & \dots & 0 & a_{2n-1} \\
0 & 0 & \dots & 0 & 0 & 0 & \dots & 0 & 1 \\
\end{pmatrix}, \; a_1, \dots, a_{2n-1} \in \F_q.
\end{align*}
Then $K$ fixes $\cW_i$ and $\cH_i$; furthermore it permutes in a single orbit the $q^{2n-1}$ points of $\Sigma_i \setminus \Pi$, $i = 1, \dots, q$. The polar hyperplane of $U_n$ with respect to $\cH_i$ is $\Pi$; if $\Gamma$ is a $(2n-3)$-space of $\Pi$ such that $U_n \notin \Gamma$ and $\Gamma \cap \Sigma_i$ is a $(2n-3)$-space, then $|\Gamma^\perp \cap \Sigma_i| = q+1$. Hence the $q^{2n-1}$ points of $\bigcup_i^{q} \Sigma_i \setminus \Pi$ lie on $q^{2n-2}$ lines of $\PG(2n-1, q^2)$ through $U_n$. Denote by $\cL$ this set of $q^{2n-2}$ lines, that is 
\begin{align*}
& \cL = \left\{\Gamma^\perp \; \mid \; \Gamma \mbox{ {\em is a} } (2n-3)\mbox{{\em-space of} } \Pi, \; U_n \notin \Gamma, \; |\Gamma \cap \Sigma_i| = \frac{q^{2n-2}-1}{q-1}\right\}.
\end{align*}
Note that $K$ acts transitively on $\cL$. In the next lemma we restrict attention on the case $n = 2$. 

\begin{lemma}\label{HermCur}
Let $\ell \in \cL$. The intersection of $\Pi$ with the tangent lines to $\cH_1$ through a point of $(\ell \cap \Sigma_i) \setminus \Pi$, $i = 2, \dots, q$, is a non-degenerate Hermitian curve of $\Pi$ belonging to the pencil determined by $\ell^\perp$ and $\Pi \cap \cH_1$. Viceversa, if $r$ is a line of $\Pi$ with $|r \cap \Sigma_i| = |r \cap \cH_1| = q+1$ and $\cU_{r}$ is a non-degenerate Hermitian curve of $\Pi$ belonging to the pencil determined by $r$ and $\Pi \cap \cH_1$, then there exists $k$ such that the lines joining a point of $(r^\perp \cap \Sigma_k) \setminus \Pi$ with a point of $\cU_{r}$ are tangent to $\cH_1$.
\end{lemma}
\begin{proof}
If $P \in (\ell \cap \Sigma_i) \setminus \Pi$, $i = 2, \dots, q$, we may assume $P = (0, a_2+\xi_i, 0,1)$, $\xi_i \ne 0$, $a_2 \in \F_q$. Then the points of $(x, y, z, 0)$ of $\Pi$ lying on a line through $P$ and tangent to $\cH_1$ satisfy 
\begin{align}
& \i^2 y^{q+1} + \i \alpha (x^q z - x z^q) = 0, \label{eqcurve} 
\end{align}
where $\alpha = \i (\xi_i^q - \xi_i)$. Hence they form a non-degenerate Hermitian curve of $\Pi$ belonging to the pencil determined by $r$ and $\Pi \cap \cH_1$. Viceversa, if $r$ is a line of $\Pi$ with $|r \cap \Sigma_i| = |r \cap \cH_1| = q+1$, we may assume that $r$ is the line joining $U_1$ and $U_3$. It follows that $\cU_r$ is given by \eqref{eqcurve}, for some $\alpha \in \F_q \setminus \{0\}$. Thus if $\alpha = \i (\xi_k^q - \xi_k)$, the lines joining a point of $(r^\perp \cap \Sigma_k) \setminus \Pi$ with a point of $\cU_{r}$ are tangent to $\cH_1$.
\end{proof}

\begin{lemma}\label{tangent}
Let $\ell$ be a line that is either tangent or contained in $\cH_1$. Then $|\ell \cap \left( \bigcup_i^{q} \Sigma_i \right)| \in \{0, 1, q+1\}$. In particular, if $|\ell \cap \left( \bigcup_i^{q} \Sigma_i \right)| = q+1$ then there exists $k$ such that $\ell \cap \Sigma_k = q+1$ and $\ell \subset \cH_k$.
\end{lemma}
\begin{proof}
If $\ell$ is contained in $\cH_1$, then there is nothing to prove. If $\ell$ is tangent, let $P_i, P_j$ be two points belonging to $\ell \cap \Sigma_i$ and $\ell \cap \Sigma_j$, respectively, where $i \ne 1$. Since $K$ is transitive on points of $\Sigma_i \setminus \Pi$, we may assume $P_i = (0, \dots, 0, \xi_i, 0, \dots, 0, 1)$. If $P_j = (a_1, \dots, a_{n-1}, a_n + \xi_j, a_{n+1}, \dots, a_{2n-1}, 1)$ for some $a_1, \dots, a_{2n-1} \in \F_q$ not all zero, and $\lambda \in \F_{q^2}$, then the point $P_j + \lambda P_i$ belongs to $\cH_1$ if and only if $\lambda$ is a root of 
\begin{align*}
\i(\xi_i - \xi_j) X^{q+1} + \i (\xi_j - \xi_i^q + a_n) X^q + \i^q (\xi_j^q - \xi_i + a_n) X + \i (\xi_j - \xi_j^q) = 0. 
\end{align*}
Some straightforward calculations show that $\ell$ is tangent if and only if $\xi_j = \xi_i - a_n$. If $i \ne j$, then $\i (\xi_j^q - \xi_j) = \i (\xi_i^q - \xi_i)$, contradicting \eqref{all}. Therefore $i = j$ and $a_n = 0$. This means that $|\ell \cap \Sigma_i| = q+1$. Moreover $\ell \subset \cH_i$. If $P_j = (a_1, \dots, a_{2n-1}, 0)$ for some $a_1, \dots, a_{2n-1} \in \F_q$, then $P_j \in \Sigma_i \cap \Sigma_j$. Hence $|\ell \cap \Sigma_i| = q+1$ and we may repeat the previous argument.  
\end{proof}

\begin{theorem}\label{tangent-set}
Let $\cO_i$ be a (partial) ovoid of $\cW_i$, $i = 1, \dots, q$, such that $\cO_i \cap \Pi = \cO_j \cap \Pi$, if $i \ne j$. Then $\cT = \bigcup_{i = 1}^{q} \cO_i$ is a tangent-set of $\cH_1$. Moreover if $n = 2$ and $\cO_i$ is an ovoid of $\cW_i$, for every $i = 1, \dots, q$, then $\cT$ is a maximal tangent-set.
\end{theorem}
\begin{proof} 
Assume by contradiction that $\cT$ is not a tangent-set. If $\ell$ is a line of $\cH_1$ and $|\ell \cap \cT| \ge 2$, then $\ell \cap \cT \subset \cO_1$, since $\cT \cap \cH_1 = \cO_1$. Hence $\ell$ is an extended line of $\cW_1$ containing two points of $\cO_1$, contradicting the assumption that $\cO_1$ is a (partial) ovoid of $\cW_1$. Similarly, if $\ell$ is tangent to $\cH_1$ and $|\ell \cap \cT| \ge 2$, then, by Lemma~\ref{tangent}, there exists a $k$ such that $\ell$ is an extended line of $\cW_k$ and $\ell \cap \cT \subset \cO_k$. Therefore $\cO_k$ is not a (partial) ovoid of $\cW_k$, since $|\ell \cap \cO_k| \ge 2$, contradicting the hypotheses. 

Assume now that $n = 2$ and $\cO_i$ is an ovoid of $\cW_i$, for every $i = 1, \dots, q$. We claim that every point $R$ of $\PG(3, q^2) \setminus \cT$ lies on at least one line that is either tangent or contained in $\cH_1$ and contains one point of $\cT$. If $R \not\in \Pi$, then there exactly one $\cH_k$ with $R \in \cH_k$; hence there is at least one extended line, say $r_k$ of $\cW_k$ containing $R$. Such a line $r_k$ has to contain exactly one point of $\cO_k$ (see Lemma~\ref{ovoid}) and by Lemma~\ref{tangent} it is contained or tangent to $\cH_1$ according as $k = 1$ or $k \ne 1$. Similarly if $R \in \cH_i \cap \Pi$. Let $R$ be a point of $\Pi \setminus \cH_i$. There are two possibilities: either $U_2 \in \cO_i$ for every $i = 1, \dots, q$, and in this case the line joining $R$ and $U_2$ is tangent to $\cH_1$ or $U_2 \not\in \cO_i$, for $i = 1, \dots, q$. If the latter possibility occurs, then there exists a line $r$ of $\Pi$ such that $R \notin r$, $|r \cap \cO_i| = 0$ and $|r \cap \Sigma_i| = |r \cap \cH_1| = q+1$. Let $\cU_r$ be the (unique) non-degenerate Hermitian curve of $\Pi$ belonging to the pencil determined by $r$ and $\Pi \cap \cH_1$ such that $R \in \cU_r$. By Lemma~\ref{HermCur}, there exists a $k$ such that the lines joining a point of $(r^\perp \cap \Sigma_k) \setminus \Pi$ with $R$ are tangent to $\cH_1$. The Baer subline $r^\perp \cap \Sigma_k$ contains two points of $\cO_k$ being the polar line of $r \cap \Sigma_k$ with respect to the symplectic polarity of $\Sigma_k$ associated with $\cW_k$. Therefore the line joining a point of $r^\perp \cap \Sigma_k \cap \cO_k$ with $R$ is tangent to $\cH_1$.  
\end{proof}

\begin{prop}
If there is a (partial) ovoid of $\cW(2n-1, q)$ of size $x+y$, with $y$ points on a hyperplane, then there exists a tangent-set of $\cH(2n-1, q^2)$ of size $x q + y$.
\end{prop}

%

The next result is reached by applying Theorem~\ref{tangent-set}, where $\cO_i$ is either an ovoid of $\cW(3, q)$, $q$ even, or the partial ovoid of $\cW(3, q)$, obtained in Theorem~\ref{partial-symp3}, or the partial ovoids of $\cW(5, q)$ described in Theorem~\ref{partial-symp5} and in Theorem~\ref{partial-symp5-even}, or the partial ovoid of $\cW(7, q)$ of size $q^3+1$ given in \cite{C, CK}.  

\begin{cor}
The following hold true.
\begin{itemize}
\item If $q$ is even, there exist maximal tangent-sets of $\cH(3, q^2)$ of size $q^3+1$ and of size $q^3-q^2+q+1$. 
\item If $q$ is an odd square and $q \not\equiv 0 \pmod{3}$, there exists a tangent-set of $\cH(3, q^2)$ of size $(q^{5/2}+3q^2-q^{3/2}+3q)/3$. 
\item If $q$ is even, there exists a tangent-set of $\cH(5, q^2)$ of size $2q^3-q^2+1$. 
\item There exists a tangent-set of $\cH(5, q^2)$ of size $q^3+q^2+1$. 
\item There exists a tangent-set of $\cH(7, q^2)$ of size $q^4+1$. 
\end{itemize}
\end{cor}

\subsection{Partial ovoids of $\cH(2n, q^2)$ from tangent-sets of $\cH(2n-1, q^2)$, $n \in \{2, 3, 4\}$}

The next result shows that starting from a (maximal) tangent-set of $\cH(2n-1, q^2)$ it is possible to obtain a (maximal) partial ovoid of $\cH(2n, q^2)$.

\begin{lemma}\label{partial-herm}
Let $P$ be a point of $\PG(2n, q^2) \setminus \cH(2n, q^2)$ and let $\cT$ be a (maximal) tangent-set of $\cH(2n-1, q^2) = P^\perp \cap \cH(2n, q^2)$. Then 
\begin{align*}
& \cO = \left( \bigcup_{R \in \cT} \langle P, R \rangle \right) \cap \cH(2n, q^2)
\end{align*}
is a (maximal) partial ovoid of $\cH(2n, q^2)$ of size $(q+1) |\cT \setminus \cH(2n-1, q^2)| +  |\cT \cap \cH(2n-1, q^2)|$.
\end{lemma} 
\begin{proof}
Assume by contradiction that $\cO$ is not a partial ovoid, then there would exist a line $\ell$ of $\cH(2n, q^2)$ containing two distinct points $P_1$ and $P_2$ of $\cO$. If at least one among $P_1, P_2$ is not in $P^\perp$, then the line $\bar{\ell} \subset P^\perp$ obtained by projecting $\ell$ from $P$ to $P^\perp$ is tangent to $\cH(2n-1, q^2) = P^\perp \cap \cH(2n, q^2)$ and contains two points of $\cT$, a contradiction. If both $P_1$ and $P_2$ are in $P^\perp$, then $\ell$ is a line of $\cH(2n-1, q^2)$ and contains two points of $\cT$, a contradiction. 

As for maximality, note that if $T$ is a point of $\cH(2n, q^2)$ such that $\cO \cup \{T\}$ is a partial ovoid of $\cH(2n, q^2)$ and $\bar{T} = \langle P, T \rangle \cap P^\perp$, then $\cT \cup \{\bar{T}\}$ is a tangent-set of $\cH(2n-1, q^2)$. 
\end{proof}

Combining Theorem~\ref{tangent-set} and Lemma~\ref{partial-herm}, the following result arises. 

\begin{theorem}
If there is a (partial) ovoid of $\cW(2n-1, q)$ of size $x+y$, with $y$ points on a hyperplane, then there exists a partial ovoid of $\cH(2n, q^2)$ of size $x q^2 + y$.
\end{theorem}

Hence starting from (partial) ovoids of $\cW(2n-1, q)$, we get partial ovoids of $\cH(2n, q^2)$, $n \in \{2, 3, 4\}$, as described below.

\begin{cor}
The following hold true.
\begin{itemize}
\item If $q$ is even, there exist maximal partial ovoids of $\cH(4, q^2)$ of size $q^4+1$ and of size $q^4-q^3+q+1$.
\item If $q$ is an odd square and $q \not\equiv 0 \pmod{3}$, there exists a partial ovoid of $\cH(4, q^2)$ of size $(q^{7/2}+3q^3-q^{5/2}+3q^2)/3$. 
\item If $q$ is even, there exists a partial ovoid of $\cH(6, q^2)$ of size $2q^4-q^3+1$. 
\item There exists a partial ovoid of $\cH(6, q^2)$ of size $q^4+q^3+1$. 
\item There exists a partial ovoid of $\cH(8, q^2)$ of size $q^5+1$. 
\end{itemize}
\end{cor}

\smallskip
{\footnotesize
\noindent\textit{Acknowledgments.}
This work was supported by the Italian National Group for Algebraic and Geometric Structures and their Applications (GNSAGA-- INdAM).
}

\end{document}